\documentclass[12pt,a4paper]{article}

\usepackage{latexsym,amssymb,setspace}

\makeindex

\begin{document}

\title{Equitable Coloring of Graphs \\ 
with Intermediate Maximum Degree}
\author{
Bor-Liang Chen\\
\normalsize Department of Business Administration\\
\normalsize National Taichung University of Science and Technology\\
\normalsize Taichung 404, Taiwan\\
\normalsize {\tt Email:blchen@ntit.edu.tw}
\and
Kuo-Ching Huang\thanks{Research supported by MOST  
(No. 103-2115-M-126-002).}\\
\normalsize Department of Financial and Computational Mathematics\\
\normalsize Providence University\\
\normalsize Shalu 43301, Taichung, Taiwan\\
\normalsize {\tt Email:kchuang@gm.pu.edu.tw}
\and
Ko-Wei Lih \thanks{Research supported by NSC 
(No. 102-2115-M-001-010).}\\
\normalsize Institute of Mathematics\\
\normalsize Academia Sinica\\
\normalsize Taipei 10617, Taiwan\\
\normalsize {\tt Email:makwlih@sinica.edu.tw}
}

\date{\small \today}

\maketitle

\newcommand{\eq}{\chi_{_{=}}}
\newcommand{\KG}{{\sf KG}}

\newtheorem{define}{Definition}
\newtheorem{proposition}[define]{Proposition}
\newtheorem{theorem}[define]{Theorem}
\newtheorem{lemma}[define]{Lemma}
\newtheorem{corollary}[define]{Corollary}
\newtheorem{problem}[define]{Problem}
\newtheorem{conjecture}{Conjecture}
%
%
\newenvironment{proof}{
\par
\noindent {\bf Proof.}\rm}%
{\mbox{}\hfill\rule{0.5em}{0.809em}\par}
\newenvironment{remark}{
\par
\noindent {\bf Remark.}\rm}

\baselineskip=19pt
\parindent=0.8cm

\begin{abstract}

\noindent
If the vertices of a graph $G$ are colored with $k$ colors such 
that no adjacent vertices receive the same color and the sizes of 
any two color classes differ by at most one, then $G$ is said to 
be equitably $k$-colorable. Let $|G|$ denote the number of vertices 
of $G$ and $\Delta=\Delta(G)$ the maximum degree of a vertex in $G$. 
We prove that a graph $G$ of order at least 6 is equitably 
$\Delta$-colorable if $G$ satisfies $(|G|+1)/3 \leqslant \Delta < 
|G|/2$ and none of its components is a $K_{\Delta +1}$.

\bigskip

\noindent
Keywords:\ {\em chromatic number; equitable coloring; 
equitable chromatic number; equitable chromatic threshold} 

\medskip

\noindent
Mathematical Subject Classification (MSC) 2010: 05C15
\end{abstract}

%
\section{Introduction}
%

A graph $G$ consists of a vertex set $V(G)$ and an edge set $E(G)$. 
All graphs considered in this paper are finite, loopless, and without 
multiple edges. Let $|G|$ denote the {\em order} of $G$, i.e., the number 
of vertices of $G$. A set of vertices of $G$ is called {\em independent} 
if its members are mutually non-adjacent. If the vertices of $G$ can be 
partitioned into $k$ subsets $V_1,V_2,\ldots ,V_k$ such that each $V_i$ 
is an independent set, then $G$ is said to be {\em $k$-colorable} and the 
$k$ sets are called {\em color classes}. Equivalently, a coloring can 
be viewed as a function $\pi: V(G)\rightarrow \{ 1,2,\ldots ,k\} $ such 
that adjacent vertices are mapped to distinct numbers. The mapping $\pi$ 
is said to be a (proper) {\em $k$-coloring}. All pre-images of a fixed $i$, 
$1 \leqslant i \leqslant k$, form a color class. The smallest integer $k$ 
such that $G$ is $k$-colorable is called the {\em chromatic number} of $G$ 
and is denoted by $\chi(G)$. The graph $G$ is said to be equitably colored 
with $k$ colors, or {\em equitably $k$-colorable}, if there is a $k$-coloring 
that satisfies the condition $| |V_i| - |V_j| | \leqslant 1$ for every pair 
of color classes $V_i$ and $V_j$. The smallest integer $k$ for which $G$ is 
equitably $k$-colorable is called the {\em equitable chromatic number} of 
$G$ and is denoted by $\eq(G)$. Clearly, $\chi(G) \leqslant \eq (G)$. Lih 
\cite{lih_new} provides a comprehensive survey of equitable coloring of graphs. 

Let $\deg_G(v)$, or $\deg(v)$ for short, denote the degree of vertex 
$v$ in $G$ and define $\Delta(G)=\max \{ \deg(v) \mid v\in V(G)\}$. 
We usually abbreviate $\Delta(G)$ to $\Delta$ when no ambiguity arises.
Let $K_n$, $P_n$ and $C_n$ denote, respectively, a complete graph, a path 
and a cycle on $n$ vertices. In 1978, Meyer \cite{m} proposed the following. 

\begin{conjecture}\label{ECC}
If a connected graph $G$ is different from a complete graph $K_n$ and an odd 
cycle $C_{2n+1}$ for any positive integer $n$, then $\eq(G)\leqslant \Delta$.
\end{conjecture}

Meyer's conjecture, if true, is a generalization of the following theorem 
of Brooks \cite{brooks}.

\begin{theorem}
If a connected graph $G$ is different from a complete graph $K_n$ and an odd 
cycle $C_{2n+1}$ for any positive integer $n$, then $\chi (G)\leqslant \Delta$.
\end{theorem}

Let $\eq^*(G)$ denote the smallest integer $m$ such that $G$ is equitably 
$k$-colorable for all $k\geqslant m$. We call $\eq^*(G)$ the {\em equitable
chromatic threshold} of $G$. The well-known Hajnal and Szemer\'{e}di Theorem 
\cite{hs} established the following for not necessarily connected graphs.

\begin{theorem} \label{haj-sze}
For a graph $G$, $\eq^*(G) \leqslant \Delta +1$. 
\end{theorem}

By definition, $\eq(G) \leqslant \eq^*(G)$. In fact, $\eq^*(G)$ may  
be greater than $\eq(G)$. For instance, the complete bipartite graph 
$K_{3,3}$ is equitably 2-colorable, but not equitably 3-colorable. 
In 1994, Chen, Lih and Wu \cite{clw} proposed the following conjecture.

\begin{conjecture}\label{EDCC}
Let $G$ be a connected graph. If $G$ is different from the complete 
graph $K_n$, the odd cycle $C_{2n+1}$, and the complete bipartite 
graph $K_{2n+1,2n+1}$ for any positive integer $n$, then $G$ is 
equitably $\Delta$-colorable.
\end{conjecture}

The conclusion of the above conjecture can be stated in an equivalent 
form $\eq^*(G)\leqslant \Delta$. It is also immediate to see that the 
Conjecture \ref{EDCC} implies Conjecture \ref{ECC}. Chen, Lih and Wu \cite{clw}
confirmed Conjecture \ref{EDCC} for the following special case. 

\begin{theorem}\label{halfdelta}
Let $G$ be a connected graph with $\Delta \geqslant |G|/2$. If $G$ is 
different from $K_n$ and $K_{2n+1,2n+1}$ for any positive integer $n$, 
then $G$ is equitably $\Delta$-colorable.
\end{theorem}

In the present paper, we are going to establish the following.

\begin{theorem}\label{main}
If a graph $G$ of order at least 6 satisfies $(|G|+1)/3 \leqslant \Delta 
< |G|/2$ and none of its components is a $K_{\Delta +1}$, then $G$ is 
equitably $\Delta$-colorable.
\end{theorem}

This implies that Conjecture  \ref{EDCC}  holds for a connected graph 
$G$ satisfying $(|G|+1)/3 \leqslant \Delta < |G|/2$. We note that 
Conjecture \ref{EDCC} has also been  established for any connected 
graph $G$ satisfying $\Delta \leqslant 3$ in \cite{clw}. Kierstead and 
Kostochka \cite{kk11} extended it to the case $\Delta =4$. Conjectures 
\ref{ECC} and \ref{EDCC} have been studied intensively with respect to 
graph classes such as forests, split graphs, outerplanar graphs,
series-parallel graphs, planar graphs, graphs with low degeneracies, 
graphs with bounded treewidth, Kneser graphs, interval graphs, etc. 
The reader is referred to \cite{lih_new} for more information.

%
\section{Main results}
%

For subsets $X$ and $Y$ of vertices of a graph $G$, let $\|X,Y\|$ 
denote the number of edges with one endpoint in $X$ and the other 
endpoint in $Y$. Clearly, $\|X,Y\|=\|Y,X\|$. We often abbreviate 
the singleton $\{x\}$ to $x$ when the context is clear. We write 
$u \sim v$ to denote that vertices $u$ and $v$ are adjacent. For 
a vertex $v\in V(G)$, we define the ({\em open}) {\em neighborhood} 
$N(v)$ to be the set $\{u \in V(G) \mid u \sim v\}$. The set $N[v]= 
N(v)\cup \{v\}$ is called the {\em closed neighborhood} of $v$. An 
$m$-independent set is an independent set of $m$ vertices. The {\em 
independence number} $\alpha(G)$ of $G$ is the maximum integer $m$ 
such that $G$ has an $m$-independent set. An $m$-{\em matching} is 
a set of $m$ mutually non-incident edges. A {\em component} of a 
graph $G$ is a maximal connected subgraph of $G$. The subgraph 
induced by a subset $S \subseteq V(G)$ is denoted by $G[S]$. The 
disjoint union of $m$ copies of a graph $G$ is denoted by $mG$. 

We call a coloring of $G$ an $[r,s,t]$-{\em coloring} if it is an 
$(r+s+t)$-coloring of $G$ having $r$ color classes of size three, 
$s$ color classes of size two and $t$ singleton color classes. The 
set of all possible $[r,s,t]$-colorings of $G$ is nonempty since 
there exists the trivial $[0,0,|G|]$-coloring.

\begin{lemma}\label{key5}
Let $G$ be a graph with $\Delta < |G|/2$. Suppose that an 
$[r,s,t]$-coloring of $G$ satisfies $r+s+t\leqslant \Delta$. 
Then, for any integer $m$, $r+s+t \leqslant m\leqslant \Delta$, 
there exists an $[a,b,c]$-coloring of $G$ satisfying $a+b+c=m$.   
\end{lemma}

\begin{proof}
Let an $[r,s,t]$-coloring of $G$ satisfying $r+s+t \leqslant \Delta$ 
have color classes $X_i=\{x_i,y_i,z_i\}$, $U_j=\{u_j,v_j\}$ and $\{w_k\}$, 
where $1 \leqslant i\leqslant r$, $1\leqslant j\leqslant s$ and $1\leqslant 
k\leqslant t$. Let $q=m-(r+s+t)$. Then $q \leqslant \Delta -(r+s+t) < r$ 
since $2\Delta < |G|=3r+2s+t \leqslant 4r+2s+2t$. We partition $X_i$ into 
$\{x_i,y_i\}$ and $\{z_i\}$ for $1 \leqslant i \leqslant q$ to obtain an 
$[r-q,s+q,t+q]$-coloring of $G$ satisfying $(r-q)+(s+q)+(t+q)=m$. 
\end{proof}

\begin{lemma}\label{key6}
Let $G$ be a graph with $\Delta < |G|/2$. Suppose that an $[r,s,t]$-coloring 
of $G$ satisfies $r+s+t\leqslant \Delta$. Then $G$ is equitably $(r+s+t)$-colorable.
\end{lemma}

\begin{proof}
Let an $[r,s,t]$-coloring of $G$ satisfying $r+s+t \leqslant \Delta$ 
have color classes $X_i=\{x_i,y_i,z_i\}$, $U_j=\{u_j,v_j\}$ and $\{w_k\}$, 
where $1 \leqslant i\leqslant r$, $1\leqslant j\leqslant s$ and $1\leqslant 
k\leqslant t$. We have $r>t$ since $r-t=(3r+2s+t)-2(r+s+t) \geqslant |G|-2
\Delta>0$. If $t=0$, we are done. Otherwise, we initiate a reduction process 
to construct an $[r-1,s+2,t-1]$-coloring of $G$. This process can be repeated 
until we obtain an $[r-t,s+2t,0]$-coloring of $G$ that is also an equitable 
$(r+s+t)$-coloring of $G$.

The reduction process is described as follows. If $w_k \not \sim w_{k'}$ 
for some $k$ and $k'$, then $X_1\cup \{w_k, w_{k'}\}$ can be partitioned 
into independent sets $\{x_1,y_1\}$, $\{z_1\}$ and $\{w_k,w_{k'}\}$. Hence, 
$G$ has an $[r-1,s+2,t-1]$-coloring. Suppose $w_k \sim w_{k'}$ for any 
distinct $k$ and $k'$. If $z_i \not \sim w_k$ for some $i$ and $k$, then 
$X_i\cup w_k$ can be partitioned into independent sets $\{x_i,y_i\}$ and 
$\{z_i,w_k\}$. Hence, $G$ has an $[r-1,s+2,t-1]$-coloring. Now suppose 
$\|X_i,w_k\|= 3$ for all $i$ and $k$. If $\|\{x_1,w_1\},U_j\| \geqslant 2$ 
for all $j$, then $2\Delta \geqslant \deg(x_1)+\deg(w_1) \geqslant 
\sum_{k=1}^{t}\|x_1,w_k\|+\sum_{i=1}^{r}\|w_1,X_i\|+\sum_{k=1}^{t}\|w_1,w_k\|
+\sum_{j=1}^{s}\|\{x_1,w_1\},U_j\| \geqslant t+3r+t-1+2s\geqslant |G| > 2\Delta$, 
a contradiction. Hence, $\|\{x_1,w_1\},U_j\| \leqslant 1$ for some $j$. Since 
$G[\{x_1,w_1\}\cup U_j]$ is equal to $P_3\cup K_1$ or $K_2\cup 2K_1$, there exist 
two disjoint 2-independent sets $A$ and $B$ in $G[\{x_1,w_1\}\cup U_j]$. Thus 
$w_1\cup X_1\cup  U_j$ can be partitioned into disjoint 2-independent sets $A$, 
$B$ and $\{y_1,z_1\}$. Hence, $G$ has an $[r-1,s+2,t-1]$-coloring.
\end{proof}

\bigskip

A coloring of $G$ is called {\em maximal} if it is an $[r,s,t]$-coloring 
of $G$ for some $r, s$ and $t$ such that for any other $[r',s',t']$-coloring, 
we have (i) $r>r'$, or (ii) $s\geqslant s'$ when $r=r'$. The existence of 
a maximal $[r,s,t]$-coloring of $G$ implies that $G$ cannot have more than 
$r$ mutually disjoint 3-independent sets.

\begin{theorem} \label{key3}
If a graph $G$ of order at least 6 satisfies $(|G|+1)/3 \leqslant \Delta 
< |G|/2$ and none of its components is a $K_{\Delta +1}$, then any maximal 
$[r,s,t]$-coloring of $G$ satisfies $r+s+t\leqslant \Delta$.
\end{theorem}

The proof of the above theorem will be deferred to the final section.

\bigskip

\noindent
{\bf Proof of Theorem \ref{main}.}
Choose any maximal $[r,s,t]$-coloring of $G$. It follows from Theorem 
\ref{key3} that $r+s+t\leqslant \Delta$. By Lemma \ref{key5}, there 
exists an $[a,b,c]$-coloring of $G$ satisfying $a+b+c= \Delta$. By 
Lemma \ref{key6}, $G$ is equitably $\Delta$-colorable. 
\mbox{}\hfill\rule{0.5em}{0.809em}

\bigskip

An examination of the proof of Theorem \ref{key3} shows that the 
following can also be derived. 

\begin{theorem}
Ift a graph $G$ of order at least 6 satisfies $(|G|+1)/3 \leqslant 
\Delta < |G|/2$ and none of its components is a $K_{\Delta +1}$, then 
\[
\eq^*(G)\leqslant \min \{r+s+t \mid \mbox{There exists an 
$[r,s,t]$-coloring of $G$}\}.
\]
\end{theorem}

%
\section{Proof of Theorem \ref{key3}}
%

\begin{lemma}\label{key1}
Let the color classes of a maximal $[r,s,t]$-coloring of $G$ be 
denoted by $X_i=\{x_i,y_i,z_i\}$, $U_j=\{u_j,v_j\}$ and $\{w_k\}$, 
where $1\leqslant i\leqslant r$, $1\leqslant j\leqslant s$ and 
$1\leqslant k\leqslant t$. Then the following statements hold.
\begin{enumerate}
\item
The vertices $w_1,w_2,\ldots ,w_t$ are mutually adjacent.
\item
For all $k$ and $j$, $\|w_k,U_j\|\geqslant 1$.
\item
If $\|w_k,U_j\|= 1$ with $w_k \sim u_j$ for some $j$ and $k$, then 
$w_{k'} \sim u_j$ for all $k'$.
\item
If $\|w_k,X_i\|= 1$ with  $w_k \sim x_i$ for some $i$ and $k$, then 
$w_{k'} \sim x_i$ for all $k'$.
\item
For all $i$, $k$ and $k'$ ($k \neq k'$), $\|\{w_k, w_{k'}\},X_i\|
\geqslant 2$. If $\|\{w_k, w_{k'}\},X_i\|=2$, then $\|w_k,X_i\|= 
\|w_{k'},X_i\|=1$.
\item
If $\|w_k,X_i\|=0$, then $\|X_i,U_j\| \geqslant 3$ and $\|w_k \cup 
X_i,U_j\|\geqslant 4$, for all $j$. Moreover, $\|w_k\cup X_i, \beta 
\| \geqslant 3$ for some $\beta$ in $U_j$.
\item
For all distinct $j$ and $j'$, there exists a 2-matching in 
$G[U_j\cup U_{j'}]$.
\item
If $\|U_j,U_{j'}\|=2$ for all distinct $j$ and $j'$, then 
$G[\cup_{h=1}^{s}U_h]=2K_s$.
\item
If $\|X_i,U_j\|=0$, then (i) \ $\|w_k,X_i\| \geqslant 2$ and $\|w_k, X_i 
\cup U_j\| \geqslant 4$ for all $k$; (ii) \  $\|\gamma,X_i\| \geqslant 2$ 
(implying $\|X_i,U_{j'}\| \geqslant 4$) and $\|\gamma,X_i\cup U_j\| 
\geqslant 4$ for all $j' \neq j$ and all $\gamma \in U_{j'}$.
\item
If $\|X_i,U_j\|=1$, then $\|X_i,U_{j'}\| \geqslant 3$ for all $j' \neq j$.
\item
For all $i$, $j$ and $j'$ ($j \neq j'$), $\|X_i,U_j\cup U_{j'}\| \geqslant 4$.
\item
If $\|X_i,U_j\|=\|X_i,U_{j'}\|=\|U_j,U_{j'}\|=2$, then $G[X_i\cup U_j\cup 
U_{j'}]=K_1\cup 2K_3$.
\item
If $\|w_k,U_j\|=\|w_k,U_{j'}\|=1$ and $\|U_j,U_{j'}\|=2$, then $G[w_k\cup U_j
\cup U_{j'}]=K_2\cup K_3$.
\end{enumerate}
\end{lemma}

\begin{proof}
{\bf 1.}\ Suppose that there were two non-adjacent $w_k$ and $w_{k'}$. 
Since $\{w_k,w_{k'}\}$ is an independent set disjoint from all $U_j$'s, 
there would be an $[r,s+1,t-2]$-coloring of $G$, a contradiction. 

{\bf 2.}\ Suppose that $\|w_k,U_j\|=0$ for some $k$ and $j$. Since 
$w_k\cup U_j$ is a 3-independent set disjoint from all $X_i$'s, 
$G$ would have more than $r$ 3-independent sets, a contradiction. 

{\bf 3.}\ Suppose that $w_{k'}$ were not adjacent to $u_j$. Since 
$\{w_k, w_{k'}\}\cup U_j$ can be partitioned into independent sets 
$\{w_k,v_j\}$ and $\{w_{k'},u_j\}$, there would be an 
$[r,s+1,t-2]$-coloring of $G$, a contradiction. 

{\bf 4.}\ Suppose that $w_{k'}$ were not adjacent to $x_i$. Since 
$\{w_k, w_{k'}\}\cup X_i$ can be partitioned into independent sets 
$\{w_k,y_i,z_i\}$ and $\{w_{k'},x_i\}$, there would be an 
$[r,s+1,t-2]$-coloring of $G$, a contradiction. 

{\bf 5.}\ Suppose that $\|\{w_k, w_{k'}\},X_i\| \leqslant 1$ for 
some $i$, $k$ and $k'$ ($k \neq k'$). Since $G[\{w_k, w_{k'}\}
\cup X_i]$ is either $P_2\cup 3K_1$ or $P_3 \cup 2K_1$, each of 
which can be partitioned into a 3-independent set and a 2-independent 
set, there would be an $[r,s+1,t-2]$-coloring of $G$, a contradiction. 
Therefore, $\|\{w_k, w_{k'}\},X_i\| \geqslant 2$ for all $i$, $k$ and 
$k'$ ($k \neq k'$). 

Now, suppose that $\|\{w_k, w_{k'}\},X_i\|=2$ for some $i$, $k$ and $k'$ 
($k \neq k'$). We may also suppose that $\|w_k,X_i\|\leqslant \|w_{k'},X_i\| 
\leqslant 2$. If $\|w_k,X_i\| \geqslant 1$, then $\|w_k,X_i\|=\|w_{k'},X_i\|
=1$ and we are done. Otherwise, $\|w_{k'},X_i\|=2$. There would be some vertex 
$\alpha \in X_i$ such that $\{w_{k'},\alpha \}$ is an independent set. Since 
$\{w_k, w_{k'}\}\cup X_i$ can be partitioned into independent sets $w_k\cup 
(X_i\setminus \{\alpha \})$ and $\{w_{k'},\alpha \}$, there would be an 
$[r,s+1,t-2]$-coloring of $G$, a contradiction. 

{\bf 6.}\ Let $\|w_k,X_i\|=0$ and $\alpha \in X_i$. Suppose that 
$\|\alpha,U_j\|=0$ for some $U_j$. Since $w_k\cup X_i\cup U_j$ can 
be partitioned into independent sets $w_k\cup (X_i \setminus \{\alpha\})$ 
and $\{\alpha\}\cup U_j$, $G$ would have more than $r$ 3-independent sets, 
a contradiction. Therefore, $\|\alpha,U_j\|\geqslant 1$ for any $\alpha 
\in X_i$ and any $U_j$, and hence $\|X_i,U_j\|\geqslant 3$. Since $\|w_k,U_j\| 
\geqslant 1$ by (2), $\|w_k\cup X_i,U_j\|=\|w_k,U_j\|+\|X_i,U_j\| \geqslant 4$.

Suppose that, for some $U_j$, $\|w_k\cup X_i,u_j\| \leqslant \|w_k\cup X_i,v_j\| 
\leqslant 2$. Since $4\leqslant \|w_k\cup X_i,U_j\|=\|w_k\cup X_i,u_j\|+\|w_k\cup 
X_i,v_j\| \leqslant 4$, we have $\|w_k,U_j\|=\|\alpha ,U_j\|=1$ for all $\alpha 
\in X_i$ by (2) and the preceding paragraph, and hence $\|w_k\cup X_i,u_j\|=\|w_k
\cup X_i,v_j\|= 2$. We may suppose that $u_j \sim w_k$, $u_j \sim x_i$, $v_j \sim 
y_i$ and $v_j \sim z_i$. Since $w_k\cup X_i\cup U_j$ can be partitioned into
independent sets $\{u_j,y_i,z_i\}$ and $\{v_j,w_k,x_i\}$, $G$ would have more 
than $r$ 3-independent sets, a contradiction.

{\bf 7.}\ Suppose that $\|u_j,U_{j'}\|=0$ in $G[U_j\cup U_{j'}]$. Since 
$u_j\cup U_{j'}$ is a 3-independent set disjoint from all $X_i$'s,     
$G$ would have more than $r$ 3-independent sets, a contradiction. Hence, 
$\|u_j,U_{j'}\| \geqslant 1$. Similarly, $\|v_j,U_{j'}\| \geqslant 1$, 
$\|u_{j'},U_j\| \geqslant 1$ and $\|v_{j'},U_j\| \geqslant 1$. Suppose that
$u_j \sim u_{j'}$. If $v_j \sim v_{j'}$, then $\{u_ju_{j'},v_jv_{j'}\}$ is 
a desired matching. Otherwise, $v_j \sim u_{j'}$. Since $\|v_{j'},U_j\| 
\geqslant 1$, $v_{j'} \sim u_j$ and then $\{u_jv_{j'},v_ju_{j'}\}$ is a 
desired matching. 

{\bf 8.}\ For $j\neq {j'}$, there is a 2-matching in $G[U_j\cup U_{j'}]$ by 
(7). Then the assumption $\|U_j,U_{j'}\|=2$ implies that $\|u_j,U_{j'}\|=
\|v_j,U_{j'}\|=1$ and $G[U_j\cup U_{j'}]=\{u_ju_{j'},v_jv_{j'}\}$ or 
$\{u_jv_{j'},v_ju_{j'}\}$. By renaming the vertices if necessary, we may 
suppose that $N[u_1]=\{u_1,\ldots ,u_s\}$ and $N[v_1]=\{v_1,\ldots ,v_s\}$ 
in $G[\cup_{t=1}^{s}U_t]$. For any distinct $j, j' > 1$, if $G[U_j\cup U_{j'}]
=\{u_jv_{j'},v_ju_{j'}\}$, then $G[U_1\cup U_j\cup U_{j'}]$ is a $C_6$ which
contains two disjoint 3-independent sets. Thus $G$ would have more than $r$ 
3-independent sets, a contradiction. Hence, $G[U_j\cup U_{j'}]=\{u_ju_{j'}, 
v_jv_{j'}\}$, and then $N[u_j]=N[u_1]$ and $N[v_j]=N[v_1]$ in $G[\cup_{h=1}^{s}U_h]$. 
Therefore, $G[\cup_{h=1}^{s}U_h]=2K_s$.

{\bf 9.}\ Assume $\|X_i,U_j\|=0$. If $\|w_k,X_i\| \leqslant 1$ for some $k$, 
then $w_k\cup (X_i\setminus \{\alpha \})$ is a 3-independent set for some 
$\alpha \in X_i$. Since $w_k\cup X_i\cup U_j$ can be partitioned into 
independent sets $w_k\cup (X_i\setminus \{\alpha \})$ and $\{\alpha \} 
\cup U_j$, $G$ would have more than $r$ 3-independent sets, a contradiction. 
Hence, $\|w_k,X_i\| \geqslant 2$ for all $k$. Similarly, $\|\gamma ,X_i\| 
\geqslant 2$ for any $j' \neq j$ and any $\gamma \in U_{j'}$, and then 
$\|X_i,U_{j'}\|=\|u_{j'},X_i\|+\|v_{j'},X_i\| \geqslant 2+2=4$.
 
Suppose that $\|w_k,X_i\cup U_j\| \leqslant 3$ for some $k$. Since 
$X_i \cup U_j$ contains exactly five vertices, there are two vertices 
$\alpha$ and $\beta$ in $X_i\cup U_j$ such that $\{w_k,\alpha,\beta\}$ 
is an independent set. Since $w_k\cup X_i\cup U_j$  can be partitioned 
into independent sets $\{w_k,\alpha,\beta\}$ and $(X_i\cup U_j)\setminus 
\{\alpha,\beta\}$, $G$ would have more than $r$ 3-independent sets, 
a contradiction. Therefore, $\|w_k,X_i\cup U_j\| \geqslant 4$ for all $k$. 
Similarly, $\|\gamma ,X_i\cup U_j\| \geqslant 4$ for any $j' \neq j$ and 
any $\gamma \in U_{j'}$. 

{\bf 10.}\ We may assume that the unique edge between $X_i$ and $U_j$ is 
$x_iu_j$. If $\|X_i,U_{j'}\| \leqslant 2$ for some $j' \neq j$, then there 
is some vertex $\alpha \in X_i$ such that $U_{j'}\cup \{\alpha \}$ is a 
3-independent set. Since $X_i \cup U_j \cup U_{j'}$ can be partitioned into  
independent sets $\{u_j\}$, $\{v_j\}\cup (X_i\setminus \{\alpha \})$ and 
$U_{j'}\cup \{\alpha \}$, $G$ would have more than $r$ 3-independent sets, 
a contradiction. Therefore, $\|X_i,U_{j'}\| \geqslant 3$ for all $j' \neq j$.

{\bf 11.}\ We may suppose that $\|X_i,U_j\| \leqslant \|X_i,U_{j'}\|$. 
If $\|X_i,U_j\|=0$, then $\|X_i,U_j\cup U_{j'}\|=\|X_i,U_{j'}\| \geqslant 
4$ by (9). If $\|X_i,U_j\|=1$, then  $\|X_i,U_j\cup U_{j'}\|=\|X_i,U_j\|+
\|X_i,U_{j'}\|\geqslant 1+3=4$ by (10). If $\|X_i,U_j\| \geqslant 2$, then 
$\|X_i,U_j\cup U_{j'}\|=\|X_i,U_j\|+\|X_i,U_{j'}\|\geqslant 2+2=4$. Therefore, 
$\|X_i,U_j\cup U_{j'}\| \geqslant 4$.

{\bf 12.}\  We may assume that $\|X_i,U_j\|=\|X_i,U_{j'}\|=\|U_j,U_{j'}\|=2$ 
with $x_i \sim u_j$, $u_j \sim u_{j'}$ and $v_j \sim v_{j'}$. If $u_j \sim y_i$ 
or $u_j \sim z_i$, then $\|v_j,X_i\|=0$. Since $\|X_i,U_{j'}\|=2$, there is some 
vertex $\alpha \in X_i$ such that $U_{j'}\cup \{\alpha \}$ is a 3-independent set. 
Since $X_i \cup U_j \cup U_{j'}$ can be partitioned into independent sets $\{u_j\}$, 
$U_{j'}\cup \{\alpha \}$ and $\{v_j\}\cup (X_i\setminus \{\alpha\})$, $G$ would 
have more than $r$ 3-independent sets, a contradiction. Hence, $\|u_j,X_i\|=1$.  

Now, suppose $v_j \sim x_i$. Since $\|X_i,U_{j'}\|=2$, there is some vertex 
$\beta \in X_i$ such that $U_{j'}\cup \{\beta \}$ is a 3-independent set. 
Let $\gamma$ denote one of $y_i$ and $z_i$ that is different from $\beta$. 
Since $\|y_i,U_j\|=\|z_i,U_j\|=0$, $U_j\cup \{\gamma\}$ is a 3-independent set. 
Since $X_i \cup U_j \cup U_{j'}$ can be partitioned into independent sets $X_i 
\setminus \{\beta,\gamma\}$, $U_{j'}\cup \{\beta \}$ and $U_j\cup \{\gamma \}$, 
$G$ would have more than $r$ 3-independent sets, a contradiction.  

Next, suppose that $v_j \sim y_i$. (The case that $v_j \sim z_i$ is similar.) 
If $u_{j'} \not \sim x_i$, since $X_i \cup U_j \cup U_{j'}$ can be partitioned 
into independent sets $\{v_{j'}\}$, $\{u_{j'},v_j,x_i\}$ and $\{u_j,y_i,z_i\}$, 
$G$ would have more than $r$ 3-independent sets, a contradiction. Hence, $u_{j'} 
\sim x_i$. If $v_{j'} \not \sim y_i$, since $X_i \cup U_j \cup U_{j'}$ can be 
partitioned into 3-independent sets $\{u_{j'}\}$, $\{u_j,v_{j'},y_i\}$ and 
$\{v_j,x_i,z_i\}$, $G$ would have more than $r$ 3-independent sets, a 
contradiction. Hence, $v_{j'} \sim y_i$. Therefore, $G[X_i\cup U_j\cup U_{j'}]$ 
consists of the singleton $z_i$ together with two 3-cycles $x_iu_ju_{j'}x_i$ and 
$y_iv_jv_{j'}y_i$.

{\bf 13.}\  We may assume that $\|w_k,U_j\|=\|w_k,U_{j'}\|=1$ and 
$\|U_j,U_{j'}\|=2$ with $w_k \sim u_j$, $u_j \sim u_{j'}$ and $v_j 
\sim v_{j'}$. Suppose that $u_{j'} \not \sim w_k$. Since $\{v_j,u_{j'},
w_k\}$ is an independent set disjoint from all $X_i$'s, $G$ would have 
more than $r$ 3-independent sets, a contradiction. Hence, $u_{j'} \sim 
w_k$, and then $G[w_k\cup U_j\cup U_{j'}]$ consists of an edge $v_jv_{j'}$ 
and a 3-cycle $w_ku_ju_{j'}w_k$.
\end{proof}

\bigskip

\noindent
{\bf Proof of Theorem \ref{key3}.}
We first note that $\Delta \geqslant 3$ when $G$ has at least 6 vertices 
and $(|G|+1)/3 \leqslant \Delta$. Then $\alpha(G) \geqslant {|G|/\chi(G)} 
\geqslant {|G|/\Delta} > 2$ by Brooks' Theorem under our assumptions. 
Hence, any maximal $[r,s,t]$-coloring of $G$ satisfies $r\geqslant 1$. 

In the first stage, we show that $r+s+t\leqslant \Delta +1$ for any maximal 
$[r,s,t]$-coloring of $G$. Suppose on the contrary that there exists a maximal 
$[r,s,t]$-coloring of $G$ with $r+s+t\geqslant \Delta +2$ such that the 
singleton color classes are $\{w_k\}$, $1\leqslant k\leqslant t$. By Theorem 
\ref{haj-sze}, there exists an equitable $(\Delta +1)$-coloring of $G$ having
$p \geqslant 0$ color classes of size $c+1$ and $q > 0$ color classes of size 
$c$. Hence $|G|=(c+1)p+cq$ and $p+q=\Delta+1$. If $c \geqslant 3$, then $|G| 
\geqslant 3\Delta+3+p > |G|$, a contradiction. If $c=2$, then $2\Delta +2+p=
|G|=3r+2s+t \geqslant 2\Delta +4+r-t$. It follow that $t\geqslant r-p+2\geqslant 
2$ by the maximality of the supposed $[r,s,t]$-coloring. If $c=1$, then $2\Delta 
+2 \geqslant \Delta +1+p=|G| \geqslant 2\Delta +4+r-t$. It follow that $t\geqslant 
r+2> 2$. Thus we always have $t \geqslant 2$. By (1), (2) and (5) of Lemma 
\ref{key1}, $2\Delta \geqslant \deg(w_1)+\deg(w_2) = \sum_{i=1}^{r} \| \{w_1,w_2\},
X_i\| + \sum_{i=1}^{s}(\|\{w_1,w_2\},U_i\|) + \sum_{i=1}^{t}(\|\{w_1,w_2\},w_i\|) 
\geqslant 2r+2s+2(t-1)=2(r+s+t-1)>2\Delta$, a contradiction. 

In the second stage, suppose that there exists a maximal $[r,s,t]$-coloring of 
$G$ with $r+s+t= \Delta +1$ such that $r \geqslant 1$ and the color classes 
are $X_i=\{x_i,y_i,z_i\}$, $U_j=\{u_j,v_j\}$ and $\{w_k\}$, where $1\leqslant 
i\leqslant r$, $1\leqslant j\leqslant s$ and $1\leqslant k\leqslant t$. Then 
we will derive contradictions for all of the following possible cases for $t$, 
and hence conclude that $r+s+t \leqslant \Delta$.

{\bf Case 1.} \ There is more than one singleton color class, i.e., $t\geqslant 2$.

Pick an arbitrary pair of distinct $k$ and $k'$. We have $2\Delta \geqslant 
\deg(w_k)+\deg(w_{k'}) = \sum_{i=1}^r\| \{w_k, w_{k'}\},X_i \| + \sum_{i=1}^s
\| \{w_k, w_{k'}\},U_i \| + \sum_{i=1}^t\| \{w_k, w_{k'}\},w_i \| \geqslant 
2r + 2s + 2(t-1)=2\Delta$ by (1), (2) and (5) of Lemma \ref{key1}. It follows 
that $\deg(w_k)=\deg(w_{k'})=\Delta$, $\|\{w_k, w_{k'}\},X_i\|=2$ and 
$\|w_k,U_j\|=1$ for all $i$, $j$, $k$ and $k'$ ($k \neq k')$. By (1) and (5) 
of Lemma \ref{key1}, we may suppose that $N[w_1]=\{x_1,\ldots ,x_r,u_1,\ldots ,
u_s,w_1, \ldots ,w_t\}$. By (3) and (4) of Lemma \ref{key1}, $N[w_k]=N[w_1]$ 
for any $k$. If $x_i \not \sim x_{i'}$ for some $i \neq i'$, then $G$ would 
have an $[r,s+1,t-2]$-coloring since $X_i\cup X_{i'}\cup \{w_1, w_2\}$ can be 
partitioned into independent sets $\{w_1,y_i,z_i\}$, $\{w_2,y_{i'},z_{i'}\}$ 
and $\{x_i,x_{i'}\}$. Hence, $x_i \sim x_{i'}$ for all $i$ and $i'$. 
Similarly, $x_i$, $u_j$ and $u_{j'}$ are mutually adjacent for all $i$, $j$ 
and $j'$ ($j \neq j'$). Then $\{x_1,\ldots ,x_r,u_1,\ldots ,u_s,w_1,\ldots ,
w_t\}$ forms a $K_{\Delta +1}$, a contradiction.

{\bf Case 2.} \ There is no singleton color class, i.e., $t=0$. 

Since $|G|=3r+2s=3\Delta +3-s$, we have $s=3\Delta +3-|G|\geqslant 4$. 

First suppose $\|X_i,U_j\| \geqslant 2$ for all $i$ and $j$. Then $2\Delta
\geqslant \deg(u_j)+\deg(v_j) = \sum_{i=1}^{r}\|X_i,U_j\| + \sum_{j'=1}^{s}%
\|U_j,U_{j'}\| \geqslant 2r+2(s-1)=2\Delta$ by (7) of Lemma \ref{key1}. 
Then $\deg(u_j)=\deg(v_j)=\Delta$, $\|X_i,U_j\|=2$  and $\|U_j,U_{j'}\|=2$  
for all $i$, $j$ and $j'$ ($j \neq j'$). By (12) of Lemma \ref{key1}, 
$\|X_i,u_j\|=\|X_i,v_j\|=1$ for all $i$ and $j$. We may suppose that $N[u_1]=
\{x_1,\ldots ,x_r,u_1,\ldots ,u_s\}$. By (8) of Lemma \ref{key1}, $\{u_1,u_2,
\ldots ,u_s\}$ forms a $K_s$. By (12) of Lemma \ref{key1}, $N[u_j]=N[u_1]$ and 
$x_i \not \sim v_{j}$ for all $i$ and $j$. If $x_i \not \sim x_{i'}$ for some 
distinct $i$ and $i'$, then $G$ would have more than $r$ 3-independent sets 
since $X_1\cup X_{i'}\cup U_1\cup U_2$ can be partitioned into independent 
sets $\{v_i\}$, $\{u_1,y_i,z_i\}$, $\{u_2,y_{i'},z_{i'}\}$ and $\{v_1,x_i,x_{i'}\}$. 
Hence, $x_i \sim x_{i'}$ for all distinct $i$ and $i'$. Then $\{x_1,\ldots ,x_r,u_1,
\ldots ,u_s\}$ forms a $K_{\Delta +1}$, a contradiction.

Next suppose $\|X_i,U_j\| \leqslant 1$ for some $i$ and $j$, say $\|X_1,U_1\| 
\leqslant 1$. Let ${\cal M}=\{X_i\mid \|X_i,U_j\| \leqslant 1 \mbox{ for some } 
j=1,2,3\}$ and $|{\cal M}|=m\geqslant 1$. If $X_i\in {\cal M}$, then $\|X_i, U_1
\cup U_2\cup U_3\|= \sum_{j=1}^{3} \|X_i,U_j\| \geqslant \min \{0+4+4,1+3+3\}=7$ 
by (9) and (10) of Lemma \ref{key1}. Therefore $6\Delta \geqslant \sum_{j=1}^{3} %
(\deg(u_j)+\deg(v_j)) = \sum_{i=1}^{r}\| X_i, U_1 \cup U_2 \cup U_3 \| + 
\sum_{i=1}^{s}\| U_i, U_1 \cup U_2 \cup U_3 \| \geqslant 7m+6(r-m)+6(s-1)= 6\Delta 
+m \geqslant 6 \Delta +1$ by (7) of Lemma \ref{key1}, a contradiction.

{\bf Case 3.} \ There is a unique singleton color class, i.e., $t=1$.
 
Since $|G|=3r+2s+1=3\Delta +1-s$, we have $s=3\Delta +1-|G|\geqslant 2$.

{\bf Subcase 3.1.}\ There exists $h$ such that $\|w_1,X_h\|=0$.

Let $A=\{X_h\mid \|w_1,X_h\|=0\}$ and $a=|A| \geqslant 1$. Pick an arbitrary 
pair of distinct $i$ and $j$. Then $4\Delta \geqslant \deg(u_i)+\deg(v_i)+
\deg(u_j)+\deg(v_j) =\sum_{X_h \in A}\|X_h, U_i \cup U_j\| + \sum_{X_h \not 
\in A}\|X_h, U_i \cup U_j\| + \sum_{h=1}^{s}\|U_h, U_i \cup U_j\| + \|w_1, 
U_i \cup U_j\| \geqslant 6a+4(r-a)+4(s-1)+ 2= 4\Delta +2(a-1)$ by (2), (6), 
(7) and (11) of Lemma \ref{key1}. Thus $a\leqslant 1$, and hence $a=1$, say 
$\|w_1,X_1\|=0$. Moreover, $\deg(u_j)=\deg(v_j)=\Delta$, $\|U_j,U_{j'}\|=2,
\|w_1,U_j\|=1$ and $\|X_1,U_j\|=3$. For each $j$, let 
\[
B_j^0=\{X_i\mid \|X_i,U_j\|=0\} \mbox{ and } b_j^0=|B_j^0|; 
\]
\[
B_j^1=\{X_i\mid \|X_i,U_j\|=1\} \mbox{ and } b_j^1=|B_j^1|.
\]
All $B_j^0$'s are mutually disjoint by (9) of Lemma \ref{key1}. All $B_j^1$'s 
are mutually disjoint by (10) of Lemma \ref{key1}. If $X_i\in B_j^0$, 
then $\|w_1,X_i\|=\|w_1,X_i \cup U_j\|-\|w_1,U_j\| \geqslant 4-1=3$ by (9) 
of Lemma \ref{key1}. Let ${\cal B}_0$ denote $\cup_{j=1}^sB_j^0$. Then 
$\Delta \geqslant \deg(w_1) = \sum_{X_i \in {\cal B}_0} \|w_1,X_i\|  + 
\sum_{X_i \not \in {\cal B}_0} \|w_1,X_i\| + \sum_{i=1}^{s}\|w_1,U_i\| 
\geqslant 3{\sum_{j=1}^{s}b_j^0} + (r-1-{\sum_{j=1}^{s}b_j^0})+s=\Delta 
+2{\sum_{j=1}^{s}b_j^0}-1$, or $2{\sum_{j=1}^{s}b_j^0}\leqslant 1$. Hence, 
$b_j^0=0$ for all $j$.

Let ${\cal B}_1$ denote $\cup_{j=1}^sB_j^1$. For an arbitrary $j$, (6), (7) 
and (10) of Lemma \ref{key1} imply that
\begin{doublespace}
\[
\begin{array}{rcl}
2\Delta &=& \deg(u_j)+\deg(v_j) \\
&=& \sum_{X_i \in B_j^1}\|X_i,U_j\| + \sum_{X_i \in {\cal B}_1 
	\setminus B_j^1}\|X_i,U_j\| \\
& & +\sum_{X_i \not \in {\cal B}_1}\|X_i,U_j\| + \sum_{j'=1}^{s} %
	\|U_j,U_{j'}\| + \|w_1,U_j\| \\
&\geqslant& b_j^1  + 3\sum_{{j'=1 \atop j' \neq j}}^sb_{j'}^1 + 3 
	+ 2(r-1- \sum_{j'=1}^{s}b_{j'}^1) + 2(s-1) + 1 \\
&=& 2\Delta +\sum_{{j'=1 \atop j' \neq j}}^sb_{j'}^1-b_j^1,
\end{array}
\] 
\noindent
equivalently, $\sum_{{j'=1 \atop j' \neq j}}^sb_{j'}^1\leqslant b_j^1.$
\end{doublespace}
By symmetry, we have either (i) $b_j^1=0$ for all $j$, or (ii) $s=2$ and 
$b_1^1=b_2^1=b>0$. In either case, for an arbitrary pair of distinct $j$ 
and $j'$, $\|X_i,U_j\|=3$ if $X_i\in B_{j'}^1$; $\|X_i,U_j\|=2$ and 
$G[X_i\cup U_j\cup U_{j'}]= K_1\cup 2K_3$ if $X_i \not \in {\cal B}_1$. 

Consider the case $b_j^1=0$ for all $j$. By (12) of Lemma \ref{key1}, 
$\|X_i,u_1\|=\|X_i,v_1\|=\|u_1,U_j\|=\|v_1,U_j\|=1$ for all $i>1$ and $j>1$. 
Then $\Delta=\deg(u_1)=\|u_1,w_1\cup X_1\|+\sum_{i=2}^{r}\|u_1, X_i\|+
\sum_{i=1}^{s}\|u_1,U_i\|=\|u_1,w_1\cup X_1\|+(r-1)+(s-1)=\|u_1,w_1\cup X_1\|+
\Delta -2$, hence $\|u_1,w_1\cup X_1\|=2$. Similarly,  $\|v_1,w_1\cup X_1\|=2$. 
These are impossible since $\|u_1,w_1\cup X_1\| 
\geqslant 3$ or $\|v_1,w_1\cup X_1\| \geqslant 3$ by (6) of Lemma \ref{key1}. 

Consider the case $s=2$ and $b_1^1=b_2^1=b>0$. Assume $j=1$ or $2$. Then 
$\|X_i,U_j\|=3$ if $X_i \in B_{3-j}^1$ by (10) of Lemma \ref{key1} and 
$G[w_1\cup U_1\cup U_2]=K_2 \cup K_3$ by (13) of Lemma \ref{key1}. We may 
let $G[w_1\cup U_1\cup U_2]=\{w_1u_1u_2w_1,v_1v_2\}$. Let 
\[
D_1=\{X_i\in B_1^1\mid \|X_i,v_1\|=1\} \mbox { and } |D_1|=d_1;
\]
\[
D_2=\{X_i\in B_2^1\mid \|X_i,v_2\|=1\} \mbox { and } |D_2|=d_2.
\]
Note that $D_1$ and $D_2$ are disjoint by (10) of Lemma \ref{key1}. Now suppose 
that $X_i\in D_j$ with $x_i \sim v_j$. If $w_1 \not \sim y_i$, then $G$ would 
have more than $r$ 3-independent sets since $w_1\cup X_i\cup U_j$ can be partitioned 
into independent sets $\{w_1,v_j,y_i\}$ and $\{u_j,x_i,z_i\}$. Hence, $w_1 \sim y_i$. 
Similarly, $w_1 \sim z_i$, $u_{3-j} \sim y_i$ and $u_{3-j} \sim z_i$. If $v_{3-j} 
\not \sim x_i$, then $G$ would have more than $r$ 3-independent sets since $X_i\cup 
U_1\cup U_2$ can be partitioned into independent sets $\{u_{3-j}\}$, $\{v_j,y_i,z_i\}$ 
and $\{u_j,x_i,v_{3-j}\}$. Hence, $v_{3-j} \sim x_i$. Then $\|w_1,X_i\| \geqslant 2$, 
$\|u_{3-j},X_i\|=2$ and $\|v_{3-j},X_i\|=1$. By the same argument, if $X_i\in B_j^1
\setminus D_j$ with $u_j \sim x_i$, then $u_{3-j} \sim x_i$, $v_{3-j} \sim y_i$ and 
$v_{3-j} \sim z_i$. Thus $\|u_{3-j},X_i\|=1$ and $\|v_{3-j},X_i\|=2$. 

Since $\Delta \geqslant \deg(w_1) \geqslant 2d_1+2d_2+(r-1-d_1-d_2)+s=\Delta 
+d_1+d_2-1$, we have $d_1+d_2\leqslant 1$. 

If $d_1+d_2=0$, then $d_1=d_2=0$. It follows that $\Delta=\deg(u_1)\geqslant 
\|u_1,w_1\cup X_1\|+(r-1)+(s-1)=\|u_1,w_1\cup X_1\|+\Delta -2$ and $\deg(v_1)
\geqslant \|v_1,w_1\cup X_1\|+ 2b_{2}^1+(r-1-b_{1}^1-b_{2}^1)+(s-1)=\|v_1,w_1
\cup X_1\|+\Delta -2$. Hence $\|u_1,w_1\cup X_1\| \leqslant 2$ and $\|v_1,w_1
\cup X_1\| \leqslant 2$, contradicting (6) of Lemma \ref{key1}. 

If $d_1+d_2=1$, say $d_1=1$ and $d_2=0$, then $\Delta=\deg(v_1)=\|v_1,w_1
\cup X_1 \|+1+2b_2^1+(r-1-b_1^1-b_2^1)+(s-1)=\|v_1,w_1\cup X_1\|+\Delta -1$, 
hence $\|v_1,w_1\cup X_1\|= 1$. Similarly, $\Delta=\deg(u_2)=\|u_2,w_1\cup X_1\|
+2+(b_1^1-1)+b_2^1+(r-1-b_1^1-b_2^1)+(s-1) =\|u_2,w_1\cup X_1\|+\Delta -1$, 
hence $\|u_2,w_1\cup X_1\|=1$. Since $G[w_1\cup U_1\cup U_2]=\{w_1u_1u_2w_1,
v_1v_2\}$, we have $\|v_1,X_1\|=\|v_1,w_1\cup X_1\|-\|v_1,w_1\|=1$ and $\|u_2,X_1\|
=\|u_2,w_1\cup X_1\|-\|u_2,w_1\|=1-1=0$. Hence, there exists some vertex $\alpha 
\in X_1$ such that $\alpha \not \sim v_1$ and $\alpha \not \sim u_2$. Since $w_1
\cup X_1\cup U_1\cup U_2$ can be partitioned into independent sets $\{u_1,v_2\}$, 
$\{v_1,u_2,\alpha \}$ and $w_1\cup (X_1\setminus \{\alpha \})$, $G$ has more than 
$r$ 3-independent sets, a contradiction.

{\bf Subcase 3.2.}\ For all $i$, $\|w_1,X_i\| \geqslant 1$.

By (2) of Lemma \ref{key1}, $\Delta\geqslant \deg(w_1)\geqslant r+s=\Delta$. 
Thus $\deg(w_1)=\Delta$ and $\|w_1,X_i\|=\|w_1,U_j\|=1$ for all $i$ and $j$. 
We may let $N(w_1)=\{x_1,\ldots ,x_r,u_1,\ldots ,u_s\}$. If $v_j \not \sim 
v_{j'}$ for some pair of distinct $j$ and $j'$, then $\{w_1,v_j,v_{j'}\}$ 
would be a 3-independent set disjoint from all $X_i$'s, a contradiction.
It follows that $\{v_1,\ldots ,v_s\}$ forms a $K_s$.

We shall establish a sequence of claims in order to show that Subcase 3.2 also 
leads to a contradiction. In the course of proving the claims, we derive one of 
the following two consequences by negating each of the claims.

(A)\  A new maximal $[r,s,t]$-coloring of $G$ is obtained such that the unique 
singleton color class is independent of some color class of size 3, i.e., Subcase 
3.1 holds.  

(B)\  More than $r$ 3-independent sets are constructed. 
 
Clearly, both (A) and (B) imply contradictions, and hence the original claims 
are true.

{\bf Claim 1.}\ For all $i$, $i'$ ($i \neq i'$) and $j$, $\deg(x_i)=\Delta$ 
and $\|x_i,X_{i'}\|=\|x_i,U_j\|=1$.

If $\|x_i,X_{i'}\|=0$ for some distinct $i$ and $i'$, then (A) occurs since 
$w_1 \cup X_i \cup X_{i'}$ can be partitioned into independent sets $\{x_i\}$, 
$X_{i'}$ and $\{w_1,y_i,z_i\}$. Hence, $\|x_i,X_{i'}\| \geqslant 1$ for all 
distinct $i$ and $i'$. If $\|x_i,U_j\|=0$ for some $i$ and $j$, then (B) occurs 
since $w_1\cup X_i \cup U_j$ can be partitioned into 3-independent sets 
$\{x_i,u_j,v_j\}$ and $\{w_1,y_i,z_i\}$. Hence, $\|x_i,U_j\| \geqslant 1$ 
for all $i$ and $j$. Therefore, $\Delta\geqslant \deg(x_i)\geqslant (r-1)+s+1=
\Delta$ and the claim is true.

{\bf Claim 2.}\ For all $i$, $j$ and $j'$ ($j \neq j'$), $\deg(u_j)=\Delta$ 
and $\|u_j,X_i\|=\|u_j,U_{j'}\|=1$.

If $\|u_j,X_i\|=0$ for some $i$ and $j$, then (A) occurs since $w_1\cup X_i\cup 
U_j$ can be partitioned into independent sets $\{u_j\}$, $X_i$ and $\{w_1,v_j\}$.
Hence, $\|u_j,X_i\| \geqslant 1$ for all $i$ and $j$. By (7) of Lemma \ref{key1}, 
$\|u_j,U_{j'}\| \geqslant 1$ for all distinct $j$ and $j'$. Therefore, $\Delta
\geqslant \deg(u_j)\geqslant r+(s-1)+1=\Delta$ and  the claim is true.

{\bf Claim 3.}\  For all $i$ and $j$, $x_i \sim u_j$.

Suppose on the contrary that $x_p \not \sim u_q$ for some $p$ and $q$. By Claim 1, 
$x_p \sim v_q$. By Claim 2, we may assume that $u_q \sim y_p$. We now prove the 
following four statements.

{\bf (3.1)}\ We have $\deg(y_p)=\Delta$, $\|y_p,X_i\|=1$ for all $i \neq p$ and 
$y_p \sim v_j$ for all $j$.

If $\|y_p,X_i\|=0$ for some $i \neq p$, then (A) occurs since $w_1\cup X_p \cup 
X_i\cup U_q$ can be partitioned into independent sets $\{y_p\}$, $X_i$, 
$\{u_q,x_p,z_p\}$ and $\{w_1,v_q\}$. Hence, $\|y_p,X_i\| \geqslant 1$ for all 
$i \neq p$. If $y_p \not \sim v_j$ for some $j$, then (B) occurs since disjoint 
3-independent sets $\{u_q,x_p,z_p\}$ and $\{w_1,v_j,y_p\}$ are included in $w_1 
\cup X_p \cup U_q \cup U_j$. Hence, $y_p \sim v_j$ for all $j$. Therefore, $\Delta
\geqslant \deg(y_p)\geqslant (r-1)+s+1=\Delta$ and the statement is true.  

{\bf (3.2)}\ We have $\deg(v_q)=\Delta$ and $\|v_q,X_i\|=1$ for all $i \neq p$.

If $\|v_q,X_i\|=0$ for some $i \neq p$, then (A) occurs since $w_1\cup U_q\cup 
X_p \cup X_i$ can be partitioned into independent sets $\{v_q\}$, $X_i$, 
$\{u_q,x_p,z_p\}$ and $\{w_1,y_p\}$. Hence, $\|v_q,X_i\| \geqslant 1$ for all 
$i \neq p$. Since $v_q$ is adjacent to $x_p$, $y_p$ and $v_j$, $\Delta \geqslant 
\deg(v_q) \geqslant 2+(r-1)+(s-1)=\Delta$ and the statement is true. 

{\bf (3.3)}\ For all $j \neq q$, $x_p \sim u_j$.

Suppose $x_p \not \sim u_j$ for some $j \neq q$. By  (3.1), $y_p \sim v_j$ 
and $\|y_p, U_h\|=1$ for all $h$, and hence $y_p \not \sim u_j$. By (3.2), 
$v_q \not \sim z_p$ since it is known that $v_q \sim x_p$. Then (B) occurs 
since disjoint 3-independent sets $\{x_p,y_p,u_j\}$ and $\{w_1,z_p,v_q\}$ 
are included in $w_1\cup X_p \cup U_q\cup U_j$. 

{\bf (3.4)}\ For all $j \neq q$, $u_q \sim u_j$.

Suppose $u_q \not \sim u_j$ for some $j \neq q$. Since $\{v_1,\ldots ,v_s\}$ 
forms a $K_s$, it follows from (3.2) that $v_q \not \sim u_j$. Then (B) occurs 
since $\{u_q,v_q,u_j\}$ is a 3-independent set disjoint from all $X_i$'s. 

Statements (3.1) to (3.4) have been established. We may choose any $q'$ 
different from $q$. By Claim 1, Claim 2, (3.3) and (3.4), $v_{q'} \not 
\sim x_p$ and $v_{q'} \not \sim u_q$. Then (B) occurs since disjoint 
3-independent sets $\{x_p,u_q,v_{q'}\}$ and $\{w_1,y_p,z_p\}$ are included 
in $w_1\cup X_p\cup U_q\cup U_{q'}$. Claim 3 is therefore proved.

{\bf Claim 4.}\ For all $i$ and $j$, $u_i \sim u_j$.

Suppose that $u_i \not \sim u_j$ for some $i$ and $j$. By Claims 1, 2 and 3, 
$x_1 \not \sim v_i$ and $\{y_1,z_1,u_i,u_j\}$ is a 4-independent set. Then 
(A) occurs since $w_1\cup X_1\cup U_i\cup U_j$ can be partitioned into 
independent sets $\{u_i\}$, $\{y_1,z_1,u_j\}$, $\{w_1,v_j\}$ and $\{x_1,v_i\}$. 
 
We have established Claims 1 to 4 and are ready
to show that a contradiction can be derived from Subcase 3.2. By Claims 3 and 4, 
$x_i \not \sim x_{i'}$ for some $i$ and $i'$ since $N(w_1)=\{x_1,\ldots ,x_r,u_1,
\ldots ,u_s\}$ and no component of $G$ is a $K_{\Delta +1}$. Then it follows from 
Claims 1, 2, 3 and 4 that (B) occurs since disjoint 3-independent sets $\{w_1,y_i,
z_i\}$, $\{u_1,y_{i'},z_{i'}\}$ and $\{v_1,x_i,x_{i'}\}$ are included in $w_1\cup 
X_i\cup X_{i'}\cup U_1$. 

Now, we have refuted Cases 1, 2 and 3 since each of them led to contradictions. 
Therefore, $G$ cannot have a maximal $[r,s,t]$-coloring with $r+s+t= \Delta +1$ 
and the proof is complete.
\mbox{}\hfill\rule{0.5em}{0.809em}

\bigskip

{\bf Acknowledgment.}\  
The authors are grateful to Professor Kostochka for directing their attention 
to a recent manuscript \cite{kk14} in which he and Kierstead established the following. 
Let $G$ be a graph with $\chi(G), \Delta(G), |G|/4 \leqslant r$. If $r$ is even 
or $G$ does not contain $K_{r,r}$, then $G$ is equitably $r$-colorable.

\bigskip 

%

\end{document}